\documentclass[12pt]{article}
\usepackage{amsmath,amsfonts,amssymb,amsthm,amscd}
\title{Remarks on compactifications of pseudofinite groups}
\date{\today}
\author{Anand Pillay\thanks{Partially supported  by NSF }\\University of Notre Dame }
\newtheorem{Theorem}{Theorem}[section]
\newtheorem{Proposition}[Theorem]{Proposition}
\newtheorem{Definition}[Theorem]{Definition} 
\newtheorem{Remark}[Theorem]{Remark}
\newtheorem{Lemma}[Theorem]{Lemma}

\newtheorem{Question}[Theorem]{Question}
\newtheorem{Conjecture}[Theorem]{Conjecture}

\newcommand{\R}{\mathbb R}

\newcommand{\N}{\mathbb N}

\begin{document}
\maketitle

\begin{abstract} 
We discuss the Bohr compactification of a pseudofinite group, motivated  by a question of Boris Zilber.  Basically referring to results in the literature we point out that  (i)  the Bohr compactification of an ultraproduct of finite simple groups is trivial, and (ii) the ``definable" Bohr compactification of any pseudofinite group $G$, relative to an ambient nonstandard model of set theory in which it is definable,  is commutative-by-profinite.

\end{abstract}

\section{Introduction}

By a pseudofinite group $G$ we mean a model of the theory of finite groups, in the group language.  
An example of a pseudofinite group is an ultraproduct of a family $G_{i}$ for  $i\in \N$, of finite groups, and every pseudofinite group is elementarily equivalent to such an ultraproduct. By a compact simple Lie group we mean a compact Lie group of positive dimension which is noncommutative and has no proper nontrivial  normal closed subgroups, other than possibly coming from a finite centre. 
In \cite{Zilber} (Section 5.3, Problem 2)  Zilber asks the following question, motivated apparently by physics:

\begin{Question}  Is there an ultraproduct $G$ of a family of finite groups $(G_{i})_{i\in \N}$, and a surjective homomorphism from $G$ to a compact simple Lie group? 
\end{Question}  

It is natural to ask the slightly ``weaker" question:

\begin{Question}
 Is there a {\em pseudofinite group}  $G$ and a surjective homomorphism from $G$ to a compact simple Lie group?
\end{Question}

In Section 3 of \cite{Zilber},  Zilber introduces a formalism of ``structural approximation" and Question 1.1 is, assuming CH,  supposed to be the  same as his question of whether a compact simple Lie group can by ``structurally approximated"  by a sequence $(G_{i}$, $i<\omega)$ of finite groups.  We will not really engage with Zilber's notion of structural approximation, but instead use the usual formalism of group  compactifications of (discrete) groups (see \cite{Auslander}, Chapter 2)  as well as its model-theoretic treatment in \cite{GPP}.   

\begin{Definition} By a {\em group compactification} of a (discrete) group $G$ we mean a compact (Hausdorff)  group $C$ and a homomorphism from $G$ into $C$ with dense image. 
\end{Definition}

In this paper we will  always understand compactifications to mean group compactfications. So an apparently even weaker question is:
\begin{Question} Is there a pseudofinite group with a compactification which is a compact simple Lie group? 
\end{Question}

\begin{Remark} Questions 1.2 and 1.4 are equivalent, in that a positive answer to one gives a positive answer to the other. Moreover, assuming some set theory such as CH, each is equivalent to Question 1.1. 
\end{Remark} 
\begin{proof} Clearly a positive answer to Question 1.2 gives a positive answer to Question 1.4. Conversely suppose that $G$ is pseudofinite and $f$ is a homomorphism from $G$ into a compact simple Lie group $C$ such that $f(C)$ is dense in $G$. Let $M_{0}$ be the structure with universe $G$, and relations the group operation on $G$ as well as all subsets of $G$. Let $M_{0}^{*}$ be a sufficiently  saturated elementary extension of $M$ and $G^{*}$ the corresponding extension of $G$. By the proof of Proposition 3.4 of \cite{GPP}, $f$ extends to a surjective homomorphism $f^{*}$ from $G^{*}$ to $C$, so as $G^{*}$ is also a pseudofinite group we get a positive answer to Question 1.2. 
\newline
Now for the moreover clause. As an ultraproduct of finite groups is pseudofinite a positive answer to Question 1.1 gives a positive answer to Question 1.2.  Now suppose  that $G$ is a pseudofinite group with a surjective homomorphism $f$ from $G$ to a compact simple Lie group. By the argument in the first part of the proof we may assume that $G$ is $\omega_{1}$-saturated, in the group language. Under CH we find an elementary substructure $H$ of $G$ which has cardinality $\aleph_{1}$, is ($\aleph_{1}$-) saturated and such that $f|H: H \to C$ is surjective.  As $H$ is pseudofinite, $H$ is elementarily equivalent to an ultraproduct of finite groups. Again assuming CH such an ultraproduct has to be saturated of cardinality $\aleph_{1}$ hence isomorphic to $H$ and we obtain a positive answer to Question 1.1. 

\end{proof} 

The point of the above discussion is to show that  Question 1.1 is really  about compactifications  in the classical sense.
Now among compactifications of a group $G$ there will be a universal one, namely a compactification $f:G\to C$ such that for every compactification $h:G\to D$ there is a unique continuous surjection $g: C\to D$ such that $h = g\circ f$.  This universal compactification is called the {\em Bohr compactification} of $G$ and denoted $bG$.

We refer to \cite{Price} for background on the structure of compact (Lie) groups, but let us mention a few key facts we will be using:  Any (connected) compact group is an inverse limit of (connected) compact Lie groups.   The connected component $C^{0}$ of a compact group $C$  is the intersection of all open subgroups of finite index (and is of course connected). The quotient $C/C^{0}$ is profinite and is the maximal profinite quotient of $C$. When $C$ is compact Lie, $C^{0}$ has finite index in $C$. A compact Lie group $C$  is defined to be {\em semisimple} if  $C$ is connected and has no positive-dimensional closed abelian normal subgroup. Semisimplicity of the compact Lie group is equivalent to $C$ being an almost direct product of finitely many compact simple Lie groups. Any connected compact Lie group is the direct product of the connected component of its centre and a semisimple compact Lie group.

\begin{Lemma} Let $G$ be any group. Then the following are equivalent.
\newline
(i) $(bG)^{0}$ is commutative,
\newline
(ii) There is no compactification $C$ of $G$ such that $C$ is a compact Lie group and $C^{0}$ is semisimple. 
\end{Lemma}
\begin{proof} If $bG$ is the inverse limit of a directed system $(L_{i})_{i}$ of compact Lie groups, then clearly $(bG)^{0}$ is the inverse limit of the $L_{i}^{0}$.  So $(bG)^{0}$ is commutative iff each $L_{i}^{0}$ is commutative iff no $L_{i}^{0}$ has a semisimple image. This suffices.

\end{proof}

So the following conjecture is equivalent to a negative answer to Question 1.4 (hence also to  Questions 1.1, 1.2) , modulo passing to connected components and allowing a finite product of compact simple Lie groups in place of a single one.
\begin{Conjecture} (Tentative) If $G$ is a pseudofinite group then $(bG)^{0}$ is commutative. 
\end{Conjecture}

We say ``tentative" in Conjecture 1.7, because a positive answer to Question 1.1 is considered to be plausible, as Zilber has informed us.

\vspace{2mm}
\noindent
In Section 2 we will prove Conjecture 1.7, but  working instead with the Bohr compactification of $G$ {\em relative to}   a natural rich structure in which the pseudofinite group is definable. This is actually a very special case of one of the main theorems of \cite{BGT}  (namely, nilpotence of good connected Lie models of ultra-approximate subgroups).  Explanations are given in the next section.   In Section 3 we prove Conjecture  1.7  when $G$ is  an ultraproduct of finite simple groups. In fact in this case we show that $bG$ is trivial, and moreover $G$ will be  absolutely connected in the sense of  Gismatullin \cite{Gismatullin}.

\vspace{2mm}
\noindent
This paper is basically a write-up of a talk given at the model theory meeting in Oaxaca, Mexico, in July 2015, but including  solutions of some questions which I  posed during the talk. Thanks to  Pierre Simon, and Boris Zilber for helpful comments and suggestions. 

\section{Definable compactifications}
We first briefly recall notions from \cite{GPP}. By ``definable" in a given structure $M$ we mean definable with parameters, unless we say otherwise.
Suppose $M$ is a first order structure and $G$ is a group definable in $M$.
By a {\em definable} (with respect to the structure $M$) group compactification of $G$ we mean a homomorphism $f:G\to C$ where $C$ is a compact group, $f(G)$ is dense in $C$ and satisfying the additional property (*) that the map $f$ is ``definable":  whenever $C_{1}, C_{2}$ are disjoint closed subsets of $C$ then there is a subset $D$ of $G$, definable in $M$ such that $f^{-1}(C_{1}) \subseteq D$ and $f^{-1}(C_{2})\subseteq G\setminus D$. 

When all subsets of $G$ happen to be definable in $M$, which we call the {\em absolute case}, then condition (*) is automatically satisfied, and so a definable compactification is just a compactification.   But even in the {\em relative case} where not all subsets of $G$ need be definable in $M$ there is always a universal defiinable (in $M$) compactification of $G$ which we call the definable Bohr compactfication of $G$, denoted $def_{M}bG$. A model-theoretic description  of the definable Bohr compactification is as follows: let $M^{*}$ be a saturated elementary extension of $M$, $G^{*}$ the interpretation of the formula defining $G$  in $M^{*}$, and $(G^{*})^{00}$ the smallest subgroup of $G^{*}$ which is type-definable over $M$ (namely defined by some conjunction of formulas with parameters from $M$) and has ``bounded index" in $G$. Then $G^{*}/(G^{*})^{00}$ equipped with the ``logic topology" and with the natural homomorphism from $G$,  coincides with $def_{M}bG$.  

Alternatively one can explicitly obtain $def_{M}bG$ without specific reference to model theory, by defining it to be the completion of $G$ with respect to the topology on $G$, whose  neighbourhoods of the identity conists of those subsets $V$ of $G$ which are definable in $M$, and admit a sequence $\{V = V_{0}, V_{1}, V_{2},...\}$  of subsets of $G$ definable in $M$ such that (i) $V_{n+1}^{2}\subseteq V_{n}$ for all $n$, and (ii) each $V_{m}$ is symmetric and ``left generic" (finitely many left translates cover $G$).

\begin{Definition} By a nonstandard finite group, we mean a finite group in the sense of some elementary extension $V^{*}$ of the standard model $V$ of set theory. Namely $G$ and the graph of its group operation are elements of $V^{*}$ and $|G|\in \N^{*}$. 
\end{Definition}

So implicit in the definition above is that a nonstandard finite group $G$ comes together with the ambient structure $V^{*} = M$ in which it is obviously definable. And $def_{M}bG$ is the ``relative" compactification of $G$ that we are interested in.  

Let us note first that an ultraproduct of finite groups ``is" a nonstandard finite group in the sense above:  Suppose $G_{i}$ for $i\in \N$ are finite groups, $U$ is an ultrafilter on $\N$ and $G = \prod_{i}G_{i}/U$ is the ultraproduct. For each $i$ let $V_{i}$ be a copy of the standard model of set theory. Let $V^{*}$ be the ultraproduct $\prod_{i}V_{i}/U$, then $G$ is an element of $V^{*}$, and of course $|G|\in \N^{*}$, so we have $G$ living canonically as a finite group in the sense of this elementary extension $V^{*}$ of the standard model.  Note that any first order formula (in the language of set theory) true of each $G_{i}$ is true of $G$ in $V^{*}$. 

And of course a nonstandard finite group is a pseudofinite group, and it is also worth remarking that any saturated pseudofinite group will have the structure of a nonstandard finite group. 

In any case we prove: 
\begin{Theorem} Let $G$ be a nonstandard finite group, in the ambient structure $M = V^{*}$. Then the connected component of $def_{M}bG$ is commutative, namely Conjecture 1.7 holds in this {\em definable} context.
\end{Theorem} 
\begin{proof}  Let $M^{*}$ be a saturated elementary extension of $M$, let $G^{*}$ the the interpretation of $G$ in $M^{*}$ and let $f:G^{*}\to def_{M}bG$ be the canonical surjective homomorphism.  So $G^{*}$ is a nonstandard finite group in the structure $M^{*}$, and  is definable over $M$. 

Now $def_{M}bG$ is an inverse limit of compact Lie groups $L_{i}$ and we want to show that $L_{i}^{0}$ is commutative for each $i$. Fix $i$ and let $L = L_{i}$, so we have an induced homomorphism $h:G\to L$, defined over $M$.  $L^{0}$ has finite index in $L$, whereby $h^{-1}(L^{0})$ is a definable (over $M$) subgroup $H$ of $G$ of finite index, which is also clearly pseudofinite.   So $h:H\to L^{0}$ is a ``good  model of $H$" in the sense of \cite{BGT}, Definition 3.5. (Note that a pseudofinite group is a special case of a pseudofinite approximate subgroup.) 
As $L^{0}$ is connected, by Theorem 9.6 of \cite{BGT}, $L^{0}$ is nilpotent. But $L^{0}$ is a compact (rather than just locally compact) connected Lie group, so $L$ is commutative.  This completes the proof. 
\end{proof}

\section{Ultraproducts of finite simple groups} 
We work back in the ``absolute" context, and prove:
\begin{Theorem} Suppose $G$ is an ultraproduct of finite simple groups. Then $bG$ is trivial.
\end{Theorem}

Note that this means the following: if $M_{0}^{*}$ is a saturated elementary extension of the structure $M_{0}$ which consists of the group $G$ with predicates for {\em all} subsets, and $G^{*}$ is the corresponding extension of $G$, then $G^{*} = (G^{*})^{00}_{M_0}$.  Now $(G^{*})^{000}$ is defined to be the smallest subgroup of $G^{*}$ which has bounded index in $G^{*}$ and is $Aut(M_{0}^{*}/M_{0})$-invariant.
Our  methods will also yield (with this notation):
\begin{Proposition} $G^{*} = (G^{*})^{000}_{M_{0}}$.  So  $G$ is {\em absolutely connected} in the sense of \cite{Gismatullin}. 
\end{Proposition}

The proof of Theorem 3.1 makes use of  results on the normal subgroup structure of ultraproducts of finite simple groups   \cite{Thomas} and  \cite{Stolz-Thom}, some of which depend on the work of Liebeck and Shalev \cite{Liebeck-Shalev}.

\begin{Definition} Let $G$ be a nonstandard finite group (living in an ambient nonstandard model $M = V^{*}$ of set theory as in Section 2). For $g\in G$, $\ell_{c} (g) = log|g^{G}|/log|G|$, where $g^{G}$ denotes the conjugacy class of $g$.  So $\ell_{c}(g)$ is in the unit interval of $\R^{*}$. 
\end{Definition}

\begin{Lemma} Let $G$ be an ultraproduct of finite simple groups, considered (as above) as a nonstandard finite group in the structure $M = V^{*}$.  Let $M^{*}$ be an elementary  extension of $M$, and $G^{*}$ the corresponding extension of $G$.  
\newline
(i)  Let $N = \{g\in G^{*}: \ell_{c}(g) < 1/n: n\in \N\}$. Then $N$ is a proper normal  subgroup of $G^{*}$ and  $G^{*}/N$ is simple  (and noncommutatve) as an abstract group.
\newline
(ii) The family  of normal subgroups of $G^{*}$ is linearly ordered by inclusion. 
\end{Lemma}
\begin{proof}  Note that $G^{*}$ is a nonstandard finite group.  Now when $G = G^{*}$ then both (i) and (ii) are contained in \cite{Stolz-Thom} and \cite{Thomas}.  So it is just a question of passing from $G$ to $G^{*}$. This follows by inspection of the proofs in the above references, and we say a few words.
\newline
(i) is precisely as in  Proposition 3.1 of \cite{Stolz-Thom}:  First as $\ell_{c}$ is (by transfer)  invariant under conjugation, $N$ is a normal subgroup of $G^{*}$.  Now if $\ell_c(g)\geq \epsilon$ for some positive (standard)  real $\epsilon$, then  $|log|G^{*}|/|log(g^{G^{*}})| \leq K = 1/\epsilon$. But by Theorem 1.1 of \cite{Liebeck-Shalev}, there is a constant $c$ for every finite simple group $H$, $\forall h\in H$ $log|H|/\log|h^{H}| \leq k$ implies that for any integer $m\geq cK$, $(h^{H})^{m} = H$.   The same is therefore true of the ultraproduct of finite simple groups $G$ in $M$, so also of $G^{*}$ in $M^{*}$, whereby $(g^{G^{*}})^{m} = G^{*}$. This shows that $G/N$ is simple (and clearly noncommutative). 
\newline
(ii) We separate into cases according to whether $G$ is an ultraproduct of alternating groups, or an ultraproduct of finite simple groups of Lie type. In the first case, $G^{*}$ is a  nonstandard finite alternating group, and  for $g\in G^{*}$ we can consider $s(g)$ which is by definition the cardinality of the support of $g$.  Now Proposition 2.4 of \cite{Thomas} says that for $g,h$ nonidentity elements of $G$, $g$ is in the normal subgroup generated by $h$ iff $s(g)/s(h)$ is finite (i.e. $< r$ for some standard positive real number $r$). The nontrivial statement is right to left. This transfers from $G$ to $G^{*}$ in the following way (or alternatively the  proof simply works for $G^{*}$): The proof of Lemma 2.7 of \cite{Thomas} gives that that if $g,h\in G$ and $s(g)/s(h) \leq k$ where $k\geq 2$ is a given  integer, then $g$ is a product of $4k$ conjugates of $h$. This transfers to $G^{*}$ (for a given integer $k$), whereby we see that for $g, h\in G^{*}$, $g$ is in the normal subgroup generated by $h$ iff $s(g)/s(h)$ is finite. This implies that the family of normal subgroups of $G^{*}$ is linearly ordered. 
\newline
In the case where $G$ is an ultraproduct of finite simple groups of Lie type, we can adapt the proof of Lemma 3.12 in \cite{Stolz-Thom} in a similar fashion. 
\end{proof} 

\vspace{5mm}
\noindent
{\em Proof of Theorem 3.1.}  
$G$ is our ultraproduct of finite simple groups. Let, as before,  $M_{0}$ be the structure consisting of $G$, its group structure and predicates for all subsets of $G$,  let $M_{0}^{*}$ be a saturated elementary extension of $M_{0}$, let $G^{*}$ be the corresponding extension of $G$,  and denote by $N_{1}$ the smallest type-definable over $M_{0}$ subgroup of $G^{*}$ of bounded index.  We must show that $N_{1} = G^{*}$. Now by saturation we may identify $G^{*}$ (as a group extending $G$) with the interpretation of the formula defining $G$ in an elementary extension $M^{*}$ of the nonstandard model $M$ of set theory in which $G$, as a nonstandard finite group lives. Now assuming that $N_{1}\neq G^{*}$, $N_{1}$ would be a proper normal subgroup of $G^{*}$, which is therefore contained in the $N$ from part (i) of Lemma 3.4. As $N_{1}$ has bounded index in $G^{*}$, $N$ also has bounded index in $G^{*}$.
Now $N$ is type-definable over $M$ in the structure $M^{*}$, whereby $G^{*}/N$ is a definable in $M$ compactification of $G$. But $G^{*}/N$ is simple (noncommutative), which contradicts Theorem 2.2.
(Actually the appeal to Theorem 2.2 should not really be necessary as one should be able to see directly that $N$ does not have ``bounded index"  in $G^{*}$.)

\vspace{5mm}
\noindent
{\em Proof of Proposition 3.2.} With notation from the proof above, suppose that $N_{2}$ were a proper $Aut(M_{0}^{*}/M_{0})$-invariant normal subgroup of $G^{*}$ of bounded index in $G^{*}$. Then $N$ would be of bounded index in $G^{*}$ again yielding a contradiction. 

\vspace{5mm}
\noindent
\begin{Remark}  Routine model-theoretic arguments allow us to conclude that if  $G$, as a group, is a model of the theory of finite simple groups, then $bG$ is trivial. Analogously for the conclusion of 3.2.
\end{Remark}

\end{document}